\documentclass[ECP]{ejpecp} 
\usepackage{mathrsfs}

\newtheorem*{theorem*}{Theorem}

\newcommand\III[1]{\mathbb{1}_{[#1]}}
\newcommand\st{\,;\;}
\newcommand\dd{\mathrm {d}}
\newcommand\all[1]{\forall #1\enspace}

\DeclareMathOperator{\Tr}{\mathrm{Tr}}
\DeclareMathOperator{\bP}{\mathbf{P}\mathopen{}}
\DeclareMathOperator{\E}{\mathbf{E}\mathopen{}}

\newcommand\Psub [1]{\bP_{\! #1}}

\newcommand\Esubbig [2]{\E_{#1}\mkern-1.5mu\bigl[#2\bigr]}
\newcommand\EsubBig [2]{\E_{#1}\mkern-1.5mu\Bigl[#2\Bigr]}
\newcommand\Esubbigg [2]{\E_{#1}\mkern-1.5mu\biggl[#2\biggr]}

\newcommand\dfn[1]{\textit{\textbf{#1}}}

\newcommand\R{{\mathbb R}}

\newcommand\N{{\mathbb N}}

\newcommand\cbuldot{{\raise.25ex\hbox{$\scriptscriptstyle\bullet$}}}

\SHORTTITLE{Exit Boundaries of Multidimensional SDEs}

\TITLE{Exit Boundaries of Multidimensional SDEs}

\AUTHORS{%
  Russell Lyons\footnote{Department of Mathematics, 831 E. 3rd St.,
  Indiana University,
  Bloomington, IN 47405-7106. \EMAIL{rdlyons@indiana.edu}. Partially
  supported by NSF grant DMS-1612363.}}


\AMSSUBJ{60H10} 

\SUBMITTED{Feb. 7, 2019}
\ACCEPTED{}

\VOLUME{0}
\YEAR{2012}
\PAPERNUM{0}
\DOI{vVOL-PID}

\ABSTRACT{We show that solutions to
multidimensional SDEs with Lipschitz coefficients and driven by Brownian
motion never reach the set where
all coefficients vanish unless the initial position belongs to that set.}

\begin{document}


The classification of isolated singular points of a 1-dimensional SDE
driven by Brownian motion is
complete and exhibits several types of behavior: see
\cite[Fig.~2.2]{Cherny} for a good summary. For example, as has long been
known, 
if $X$ is a (weak)
solution to $E_x(\sigma, 0)$ with $\sigma^{-2}$ being nonzero and locally
integrable in some interval $(0, a]$ and $x \in (0, a)$, then the
probability that $X_t$ ever reaches 0 is positive (i.e., 0 is \dfn{accessible})
iff $\int_0^a y\, \sigma(y)^{-2} \,\dd y < \infty$.  Much less is known in
higher dimensions. In particular, the following theorem that makes the
usual assumption of Lipschitz coefficients seems to be new:

\begin{theorem*} 
Let $d, m \in \N^+$. Let $B = (B^{(1)}, \dots, B^{(m)})$ be $m$-dimensional
Brownian motion.
Let $\sigma \colon \R^d \to M_{d \times m}(\R)$ and $b \colon \R^d \to
\R^d$ be Lipschitz. Write 
$$
\Lambda := \{x \in \R^d \st \sigma(x) = 0,\, b(x) = 0\}.
$$
Suppose that 
$X$ solves $E_x(\sigma, b)$, i.e.,
$$
X_t = x + \int_0^t \sigma(X_s)\, \dd B_s + \int_0^t b(X_s)\, \dd s \quad (t
\ge 0).
$$
If $x \notin \Lambda$, then
$$
\bP[\all {t \ge 0} X_t \notin \Lambda] = 1.
$$
In other words, the set $\Lambda$ is inaccessible.
\end{theorem*}

\begin{proof}
We use the Frobenius norm $\|M\| := \sqrt{\Tr(M^* M)}$ for a matrix, $M$.
For $A > 0$, define the stopping time
\[
T_A
:=
\inf \bigl\{ t \ge 0 \st \|\sigma(X_t)\|^2 + \|b(X_t)\|^2 = A\bigr\}.
\]
Fix $A > 0$.
For $k \in \N^+$, write
\[
S_k := T_{A/2^{k+1}} \wedge T_{A/2^{k-1}}.
\]
If $x$ is such that $\|\sigma(x)\|^2 + \|b(x)\|^2 = A/2^k$, then $\forall t
\ge 0$
\begin{align*}
\Esubbig x{\|x - X_{t \wedge S_k}\|^2 &\III{S_k \le 1}}
\le
(m+1) \E_x\mkern-1.5mu\biggl[\sum_{i=1}^d \sum_{j=1}^m
\Bigl(\int_0^{t \wedge S_k} \sigma(X_u)_{i, j} \,\dd B_u^{(j)}\Bigr)^2
\\ \shortintertext{$ \displaystyle \hfill {}+ \sum_{i=1}^d
\Bigl(\int_0^{t \wedge S_k} b(X_u)_i \,\dd u\Bigr)^2 \III{S_k \le 1}\biggr]$}
&\le
(m+1) \Esubbigg x{\int_0^{t \wedge S_k} \|\sigma(X_u)\|^2 \,\dd u}
   + (m+1)\Esubbigg x{\int_0^{t \wedge S_k} \|b(X_u)\|^2 \,\dd u}
\\ &\le
(m+1) \cdot \frac{A}{2^{k-1}} \cdot t
\end{align*}
and
\begin{align*}
\EsubBig x{&\bigl| \|\sigma(x)\|^2 + \|b(x)\|^2 - \|\sigma(X_{t \wedge
S_k})\|^2 - \|b(X_{t \wedge S_k})\|^2 \bigr| \st S_k \le 1}
\\ &\le
\E_x\mkern-1.5mu\Bigl[\|\sigma(x) + \sigma(X_{t \wedge S_k})\| \cdot
\|\sigma(x) - \sigma(X_{t \wedge S_k})\| 
\\ \shortintertext{$ \displaystyle \hfill {}+ 
\|b(x) + b(X_{t \wedge S_k})\| \cdot
\|b(x) - b(X_{t \wedge S_k})\| \st S_k \le 1\Bigr]$}
&\le
\EsubBig x{\bigl(\|\sigma(x)\| + \|\sigma(X_{t \wedge S_k})\| +
\|b(x)\| + \|b(X_{t \wedge S_k})\|\bigr) \cdot K \cdot
\|x - X_{t \wedge S_k}\| \st S_k \le 1}
\\ &\le
2 \cdot \Bigl(\frac{A + 2A}{2^k}\Bigr)^{1/2} \cdot K \cdot 
\Esubbig x{\|x - X_{t \wedge S_k}\|^2 \st S_k \le 1}^{1/2},
\end{align*}
where $K$ is a bound for the Lipschitz constants.
If, in addition, $t \le 1$ and $S_k \le t$, then
\[
\bigl|\|\sigma(x)\|^2 + \|b(x)\|^2 - \|\sigma(X_{t \wedge S_k})\|^2 - \|b(X_{t
\wedge S_k})\|^2 \bigr| \III{S_k \le 1} 
\ge
\frac{A}{2^{k+1}}.
\]
Putting these inequalities together, we obtain
$\forall{t \le 1}$
\[
\Psub x[S_k \le t]
\le
\frac{2^{k+1}}{A} \cdot 2\Bigl(\frac{3A}{2^k}\Bigr)^{1/2} \cdot K \cdot
\sqrt{(m+1) \cdot \frac{A}{2^{k-1}} \cdot t}
=
C \sqrt t
\]
for some constant, $C$, depending only on $m$ and $K$.

Choose $t_0 \in (0, 1)$ so that $C \sqrt{t_0} \le 1/2$. Then by the strong
Markov property, if
$k \in \N^+$, $\>A > 0$, and $x \in \R^d$,
\[
 \|\sigma(x)\|^2 + \|b(x)\|^2 \ge A/2^k \quad \Longrightarrow \quad
\Psub x[T_{A/2^{k+1}} \ge t_0] \ge 1/2.
\]
Given $x \notin \Lambda$,
choose $A := \|\sigma(x)\|^2 + \|b(x)\|^2$ and express the
time to reach $\Lambda$ as $\sum_{k \ge 0} \bigl(T_{A/2^{k+1}} -
T_{A/2^{k}}\bigr)$. By
the strong Markov property, infinitely many of these terms are at least
$t_0$ a.s., whence the total time is infinite a.s.
\end{proof}

\textbf{Acknowledgment.} We are grateful to
Jean-Fran\c cois Le Gall for showing us a
similar idea in a different context.




\providecommand{\bysame}{\leavevmode\hbox to3em{\hrulefill}\thinspace}
\providecommand{\MR}{\relax\ifhmode\unskip\space\fi MR }
\providecommand{\MRhref}[2]{%
  \href{http://www.ams.org/mathscinet-getitem?mr=#1}{#2}
}
\providecommand{\href}[2]{#2}



\end{document}